\documentclass[preprint]{article}
\usepackage[utf8]{inputenc}

\usepackage{amsmath,amsthm,amssymb}
\usepackage{color}
\usepackage{cite}
\usepackage{authblk}
\usepackage{booktabs}

\usepackage{tikz}
\usepackage{subdepth}
\usepackage{enumerate}
\usepackage{mathtools}

\usepackage{hyperref}

\newtheorem{lemma}{Lemma}
\newtheorem{theorem}[lemma]{Theorem}
\newtheorem{corollary}[lemma]{Corollary}

\theoremstyle{remark}
\newtheorem{remark}[lemma]{Remark}

\theoremstyle{example}

\theoremstyle{definition}
\newtheorem{definition}[lemma]{Definition}

\newcommand{\CC}{\mathbb{C}}

\newcommand{\cS}{{\cal S}}
\newcommand{\cQ}{{\cal Q}}
\newcommand{\cP}{{\cal P}}

\newcommand{\wt}{\widetilde}

\newcommand{\la}{\lambda}

\newcounter{algo}[section]

\newcommand{\onetwo}[2]{\left[\begin{array}{cc} #1 & #2 \end{array}\right]}

\renewcommand{\leq}{\leqslant}
\renewcommand{\geq}{\geqslant}

\renewcommand{\theta}{\vartheta}

\DeclareMathOperator{\opvec}{vec}

\DeclareMathOperator{\re}{re}
\DeclareMathOperator{\im}{im}

\DeclareMathOperator{\rev}{rev}

\newcommand{\C}{\mathbb{C}}
\DeclarePairedDelimiter{\abs}{\lvert}{\rvert}

\title{Solvability and uniqueness criteria for generalized Sylvester-type equations\thanks{This work was partially
supported by the Ministerio de Econom\'{i}a y Competitividad of Spain
through grants MTM2015-68805-REDT, and MTM2015-65798-P  (F. De Ter\'an), and by an INdAM/GNCS Research Project 2016 (B.~Iannazzo, F.~Poloni, and L.~Robol). Part of this work was done during a visit of the first author to the Universit\`a di Perugia as a Visiting Researcher.}
}

\author[1]{Fernando De Ter\'an}
\author[2]{ Bruno Iannazzo}
\author[3]{Federico Poloni}
\author[4]{Leonardo Robol} 

\affil[1]{\footnotesize Departamento de Matem\'aticas, Universidad Carlos III de Madrid, Avda. Universidad 30, 28911 Legan\'es,
Spain. {\tt fteran@math.uc3m.es}. Corresponding author.}
\affil[2]{ Dipartimento di Matematica e Informatica, Universit\`a di Perugia, Via Vanvitelli 1, 06123 Perugia, Italy. {\tt bruno.iannazzo@dmi.unipg.it}.}
\affil[3]{ Dipartimento di Informatica, Universit\`a di Pisa, Largo B. Pontecorvo 3, 56127 Pisa, Italy. {\tt federico.poloni@unipi.it}.}
\affil[4]{ Dept. Computerwetenschappen, KU Leuven, Celestijnenlaan 200A, 3001 Heverlee (Leuven), Belgium. {\tt leonardo.robol@cs.kuleuven.be}.}

\begin{document}

\maketitle

\begin{abstract} 
We provide necessary and sufficient conditions for the generalized $\star$-Sylvester matrix equation, $AXB+CX^{\star}D=E$, to have exactly one solution for any right-hand side $E$. These conditions are given for arbitrary coefficient matrices $A,B,C,D$ (either square or rectangular) and generalize existing results for the same equation with square coefficients.
We also review the known results regarding the existence and uniqueness of solution for generalized Sylvester and $\star$-Sylvester equations.
\end{abstract}


\medskip


\section{Introduction}

We consider the {\em generalized $\star$-Sylvester equation} 
\begin{equation}\label{gensylv}
AXB+CX^\star D=E
\end{equation}
for the unknown $X\in\CC^{m\times n}$, with $\star$ being either the transpose ($\top$) or the conjugate transpose ($*$), and $A,B,C,D,E$ being matrices with appropriate sizes. We are interested in the most general situation, where both the coefficients and the unknown are allowed to be rectangular. This equation is closely related to the {\em generalized Sylvester equation}
 \begin{equation}\label{eq:gensylvr}
AXB-CXD=E,
\end{equation}
and equations \eqref{gensylv} and \eqref{eq:gensylvr} are natural extensions of the {\em$\star$-Sylvester equation} and the {\em Sylvester equation}, $AX+X^\star D=E$ and $AX-XD=E$, respectively.

Sylvester-like equations are among the most popular matrix equations, and they arise in many applications (see, for instance, \cite{bauapp,braapp,bk06,ksw09} and the recent review \cite{simoncini}). In particular, equations with rectangular coefficients arise in several eigenvalue perturbation and updating problems \cite{kmnt14,yuanapp}.

As with every class of equations, the two most natural questions regarding~\eqref{gensylv} and~\eqref{eq:gensylvr} are:
\begin{description}
 	\item[S (solvability).] Does the equation have a solution, for given $A,B,C,D,E$?
 	\item[US (unique solvability).] Does the equation have exactly one solution, for given $A,B,C,D,E$? 
\end{description} 

Moreover, three additional questions arise naturally for linear equations, due to their peculiar structure:
\begin{description}
	\item[SR (solvability for any right-hand side).] Given $A,B,C,D$, does the equation have \emph{at least} one solution for any choice of the right-hand side $E$? 
	\item[OR (at most one solution for any right-hand side).] Given $A,B,C,D$, does the equation have \emph{at most} one solution for any choice of the right-hand side $E$?
	\item[UR (unique solvability for any right-hand side).] Given $A,B,C,D$, does the equation have \emph{exactly} one solution for any choice of the right-hand side $E$?
\end{description}

Though the contribution of this paper is mainly restricted to question {\bf UR}, we present here an overview of the known results on all these problems for equations \eqref{gensylv} and \eqref{eq:gensylvr}. The main results regarding these closely related problems are scattered within the literature, so we gather all them in Table \ref{table:results} (the contents of this table are explained below). 

Using vectorizations, equations~\eqref{gensylv} and~\eqref{eq:gensylvr} can be transformed into a system of linear equations (either over the real or the complex field) of the form $Mx=e$, where $M$ depends only on the coefficients $A,B,C,D$, and $e$ depends only on the right-hand side $E$. This is clear for~\eqref{eq:gensylvr}, where $M = B^\top \otimes A - D^\top \otimes C$ (see e.g.~\cite[Section~4.3]{hj-topics}), and can be established with a little more effort for~\eqref{gensylv} (see Section~\ref{size-sec}). In this setting, question~\textbf{SR} is equivalent to asking whether $M$ has full row rank, and question~\textbf{OR} is equivalent to asking whether $M$ has full column rank. Question~\textbf{OR} is also equivalent to asking whether the homogeneous equation ($E=0$) has only the trivial solution $X=0$. In order for the answer of question~\textbf{UR} to be affirmative, $M$ must be square and, in this case, all \textbf{US}, \textbf{SR}, and \textbf{OR} are equivalent 
 to \textbf{UR}.

Instead of the conditions on the matrix $M$, in many cases it is possible to find conditions of a different kind, related to the spectral properties of smaller matrix pencils. For instance, let us consider question \textbf{UR} in the case in which all matrices are square and have the same size $m=n$: for Equation~\eqref{eq:gensylvr}, it has a positive answer if and only if the $n\times n$ matrix pencils $A - \la C$ and $D - \la B$ are regular and do not have common eigenvalues~\cite[Th. 1]{chu87}; and for Equation~\eqref{gensylv} the answer depends on spectral properties of the $2n\times 2n$ matrix pencil
\begin{equation}\label{pencil}
\cQ(\la)=\left[\begin{array}{cc}\la D^\star&B^\star\\A&\la C\end{array}\right]
\end{equation}
(see Theorem \ref{thm:di16} for a precise statement, or~\cite[Th. 15]{di16} for more details).
By contrast, $M$ has size $n^2$ or $2n^2$. Even for equations with rectangular coefficients, the picture is the same: the characterizations that do not involve vectorization lead to matrix pencils with smaller size, of the same order as the coefficient matrices, while approaches based on Kronecker products lead to larger dimensions.

Several authors have given conditions of this kind: in Table~\ref{table:results}, we show an overview of these results. In this table we consider separately the cases where the coefficient matrices $A,B,C,D$ are square, and the more general case where they have arbitrary size (as long as the product is well-defined). This distinction fits with the historical flow of the problem, as can be seen in the table. Note that, in both \eqref{gensylv} and \eqref{eq:gensylvr}, if the coefficient matrices $A,B,C,D$ are square then the coefficient matrix $M$ of the associated linear system is square as well. However, there is a difference between \eqref{gensylv} and \eqref{eq:gensylvr} regarding this issue: whereas in \eqref{eq:gensylvr} it may happen that all $A,B,C,D$ are square but $A$ and $B$ have different size, in \eqref{gensylv} if all coefficient matrices are square, then they must all have the same size in order for the products to be well-defined.

\begin{table}
\centering
\begin{tabular}{c|cc|ccc}
\toprule
&\multicolumn{2}{c|}{$AXB-CXD=E$} & \multicolumn{2}{c}{$AXB+CX^\star D=E$}\\
&\begin{tabular}{c}square \\coefficients\end{tabular} & \begin{tabular}{c}general\\ coefficients\end{tabular}& \begin{tabular}{c}square \\coefficients\end{tabular} & \begin{tabular}{c}general\\ coefficients\end{tabular}\\
\midrule\textbf{S} & \cite[Th. 6.1]{dk16} & \cite[Th. 6.1]{dk16}, \cite[Th. 1]{kosir-tech} & \cite[Th. 6.1]{dk16} & \cite[Th. 6.1]{dk16}\\
\textbf{US} & \cite[Th. 1]{chu87} & \cite[Th. 1]{kosir-tech}
  & \cite[Th. 15]{di16} & open \\
\textbf{SR} & same as~\textbf{US} &Th.~\ref{thm:kosir_characterization} (using~\cite{kosir-tech})
 & same as~\textbf{US} & open \\
\textbf{OR} & same as~\textbf{US} &\cite[Cor. 5]{kosir},Th.~\ref{thm:kosir_characterization} (using~\cite{kosir-tech})& same as~\textbf{US} & open \\
\textbf{UR} & same as~\textbf{US} & Th.~\ref{thm:kosir_characterization} (using~\cite{kosir-tech})& same as~\textbf{US} & Th.~\ref{thm:main} \\
\bottomrule
\end{tabular}
\caption{\footnotesize Existing solvability and uniqueness results for Equations~\eqref{gensylv} and~\eqref{eq:gensylvr} in terms of ``small-size'' matrices and pencils.} \label{table:results}
\end{table}

The solution of Equation~\eqref{eq:gensylvr}, allowing for rectangular coefficients, had been considered in \cite{mitra}, where an approach through the Kronecker canonical form of $A-\la C$ and $D-\la B$ was proposed. However, no explicit characterization of the uniqueness of solution was given in that reference. Also, \cite[Th. 3.2]{kmnt14} presents a computation of the solution space of \eqref{eq:gensylvr} with $B=I$, depending on the Kronecker canonical form of $A-\la C$, but restricted to the case where this canonical form does not contain right singular blocks. In \cite{rozsa}, the author identifies the correct line of attack of the problem, while a complete analysis of the solution space of Equation \eqref{eq:gensylvr} is given by Ko\v{s}ir in \cite{kosir-tech}; the same author gives an explicit answer to question~\textbf{SR} in~\cite{kosir}. Answers to~\textbf{OR} and~\textbf{UR} follow from Ko\v{s}ir's work \cite{kosir-tech}, but they are not explicitly stated there; for completeness, we formulate them in Section~\ref{sec:kosir_criteria}.

Instead, for the $\star$-Sylvester equation~\eqref{gensylv}, to the best of our knowledge, several problems are still open; i.e., only characterizations based on vectorization and Kronecker products are known. The main goal of this paper is to give a solution to~\textbf{UR} in the most general case of rectangular coefficients: we give necessary and sufficient conditions for the unique solvability of~\eqref{gensylv} in terms of the pencil $\mathcal{Q}(\lambda)$ in~\eqref{pencil}.

Our focus on question~\textbf{UR} is motivated by the fact that this is the only case where the operator $X\mapsto AXB+CX^\star D$ is invertible. Moreover, this is the only case where the solution of the equation is a well-posed problem, since the (unique) solution depends continuously on the (entries of the) coefficient matrices. This is no longer true in the remaining cases.

We emphasize that the approach followed in \cite{di16} for square coefficients cannot be applied in a straightforward manner to the rectangular case. Indeed, that approach is based on the characterization of the uniqueness of solution of the $\star$-Sylvester equation $AX+X^\star D=0$ provided in \cite{bk06,ksw09}, which is valid only for $A,D,X\in\CC^{n\times n}$. We follow a different approach that allows us to extend that characterization to the case of rectangular coefficients.

The most interesting feature of our characterization is the appearance of an additional invertibility constraint that is not present in the square case. 
Indeed, the unique solvability of~\eqref{gensylv} cannot be characterized completely in terms of the eigenvalues of $\mathcal{Q}(\lambda)$ only, as we show with a counterexample in Section~\ref{sec:counterexample}.

\subsection{The main result}\label{sec:mainth}

In this section we state the main result of this paper, namely Theorem \ref{thm:main}. The rest of the paper is devoted to prove this result, but, before its statement, we introduce some notation and tools. 

Throughout the paper we denote by $I$ the identity matrix of appropriate size, and by $M^{-\star}$ we denote the inverse of the matrix $M^\star$, for an invertible matrix $M$.

A matrix pencil $\cP(\la)=\la M+N$ is said to be singular if either $\cP(\la)$ is rectangular or $q(\la):=\det\bigl(\cP(\la)\bigr)$ is identically zero. If $\cP(\la)$ is not singular, and so $M,N$ are $n 
\times n$ matrices, then it is said to be regular and the set of roots of $q(\la)$, complemented with $\infty$ if the degree of $q(\la)$ is less than $n$, is the spectrum of $\cP$, denoted by $\Lambda(\cP)$.
With $m_\la(\cP)$ we denote the algebraic multiplicity of the eigenvalue $\la$ in $\cP$. 

We shall deal with certain matrices and matrix pencils that always have $\abs{m-n}$ zero or infinite eigenvalues which are \emph{dimension-induced}, that is, they are present simply because of the sizes of the coefficient matrices they are constructed from (see~\cite{sergeichuk2004computation}). Hence we define a variant of the spectrum in which these eigenvalues are omitted:
\[
\widehat \Lambda(\cP):=\left\{\begin{array}{cc}\Lambda(\cP),&\mbox{if $m_\infty(\cP)> |m-n|$,}\\ \Lambda(\cP)\setminus\{\infty\},&\mbox{if $m_\infty(\cP)=|m-n|$,}\end{array}\right.
\]
\[
  \widetilde \Lambda(\cP):=\left\{\begin{array}{cc}\Lambda(\cP),&\mbox{if $m_0(\cP)> |m-n|,$}\\ \Lambda(\cP)\setminus\{0\},&\mbox{if $m_0(\cP)=|m-n|.$}\end{array}\right.
\]
Following~\cite{sergeichuk2004computation}, we refer to the
eigenvalues in either $\widehat \Lambda(\cP)$ or
$\widetilde \Lambda (\cP)$ as~\emph{core eigenvalues}.

The {\em reversal pencil} of the matrix pencil $\cP(\la)=\la M+N$ is the pencil $\rev \cP(\la):=\la N+M$. The pencil $\cP(\la)$ has an infinite eigenvalue if and only if $\rev \cP(\la)$ has the zero eigenvalue. The multiplicity 
of the infinite eigenvalue in $\cP(\la)$ is the multiplicity of the zero eigenvalue in $\rev \cP(\la)$, thus
\begin{equation}\label{eq:rev}
\wt \Lambda(\rev\mathcal P)=\left\{\la^{-1}\ :\ \la\in\widehat\Lambda(\mathcal P)\right\},
\end{equation}
with $0^{-1}=\infty$ and $\infty^{-1}=0$.

By $\la^\star$ we denote either $\la$, if $\star=\top$, or $\overline \la$, if $\star=*$, with $\overline \la$ being the complex conjugate of $\la$.

If $M$ is a square matrix, by $\Lambda(M)$ and $m_\la(M)$ we denote, respectively, the spectrum of $M$ and the algebraic multiplicity of $\la$ as an eigenvalue of $M$. Furthermore, we use $\widetilde \Lambda(M)$ to denote $\widetilde \Lambda(\lambda I - M)$.

We also recall the following notion, which plays a central role in Theorem \ref{thm:main}. 

\begin{definition}\label{def:reciprocal} {\rm(Reciprocal free and $*$-reciprocal free set) \cite{bk06,ksw09}}. Let $\cS$ be a subset of $\CC\cup\{\infty\}$. We say that $\cS$ is 
\begin{itemize} 
\item[{\rm(a)}] {\rm reciprocal free} if $\la\neq \mu^{-1},$ for all $\la,\mu\in \cS$; 
\item[{\rm(b)}]  {\rm $*$-reciprocal free} if $\la\neq (\overline \mu)^{-1},$ for all $\la,\mu\in \cS$.
\end{itemize}
This definition includes the values $\la=0,\infty$, with the customary assumption $\la^{-1}=(\overline\la)^{-1}=\infty,0$, respectively.
\end{definition}

Before stating the characterization of the uniqueness of solution in the general case,
we recall here the main result in \cite{di16}, namely the characterization of the uniqueness of solution of \eqref{gensylv} when all coefficients are square and have the same size.

\begin{theorem}\label{thm:di16}{\rm\cite[Th. 15]{di16}}
Let $A,B,C,D\in\CC^{n\times n}$ and let 
$\cQ(\la)=\left[\begin{smallmatrix} \la D^\star & B^\star\\A & \la C\end{smallmatrix}\right]$.
Then the equation 
$AXB+CX^\star D= E$
has a unique solution, for any right-hand side $E$, if and only if $\cQ(\la)$ is regular and:
\begin{itemize} 
\item If $\star=\top$, $\Lambda(\cQ)\setminus\{\pm1\}$ is reciprocal free and $m_1(\cQ)=m_{-1}(\cQ)\leq1$.
\item If $\star=*$, $\Lambda(\cQ)$ is $*$-reciprocal free.
\end{itemize}
\end{theorem}

If we allow for rectangular coefficient matrices, that is
$A\in\C^{p\times m}, B\in\C^{n\times q}, C\in\C^{p\times
  n},D\in\C^{m\times q}$, then several subtleties arise, and in the
end, they will result in additional restrictions on the pencil
$\cQ(\la)$. More precisely, we will prove that the only case in which
unique solvability arises is when $(p,q)\in\{(m,n),(n,m)\}$. In
the case where $p=m$, that is,
$A\in\CC^{m\times m},B\in\CC^{n\times n},C\in\CC^{m\times
  n},D\in\CC^{m\times n}$, the spectrum of the matrix pencil
\eqref{pencil} contains the infinite eigenvalue, with multiplicity at
least $|m-n|$. Then, in this case we denote by $\widehat \Lambda(\cQ)$
the following set obtained from $\Lambda(\cQ)$:
\[
\widehat \Lambda(\cQ):=\left\{\begin{array}{cc}\Lambda(\cQ),&\mbox{if $m_\infty(\cQ)> |m-n|$,}\\ \Lambda(\cQ)\setminus\{\infty\},&\mbox{if $m_\infty(\cQ)=|m-n|$.}\end{array}\right.
\]
If $p=n$, that is
$A\in\CC^{n\times m},B\in\CC^{n\times m},C\in\CC^{m\times m},D\in\CC^{n\times n}$, the spectrum of the matrix pencil \eqref{pencil} contains the zero eigenvalue, with multiplicity at least $|m-n|$. Then, in this case we denote by $\widetilde \Lambda(\cQ)$ the following set obtained from $\Lambda(\cQ)$:
\[
\widetilde \Lambda(\cQ):=\left\{\begin{array}{cc}\Lambda(\cQ),&\mbox{if $m_0(\cQ)> |m-n|,$}\\ \Lambda(Q)\setminus\{0\},&\mbox{if $m_0(\cQ)=|m-n|.$}\end{array}\right.
\]

The presence of these additional zero/infinity eigenvalues of $\cQ(\la)$ in \eqref{pencil} is due to the ``rectangularity" of either the diagonal blocks $C,D$ or the anti-diagonal blocks $A$ and $B$. Following \cite{granat2007matlab}, based on the theory developed
in \cite{sergeichuk2004computation}, these extra zero/infinity eigenvalues are called \emph{dimension induced} eigenvalues. The sets $\widehat \Lambda(\cQ)$ and $\widetilde \Lambda(\cQ)$ are referred to as the set of \emph{core eigenvalues}.

With these considerations in mind, we can state the main result of this paper, which is an extension of Theorem \ref{thm:di16} and which will be proved in Section~\ref{sec:starsylv}.

\begin{theorem}\label{thm:main}
Let $A\in\CC^{p\times m},B\in\CC^{n\times q},C\in\CC^{p\times n},$ and $D\in\CC^{m\times q}$ and set $\cQ(\la):=\left[\begin{smallmatrix} \la D^\star & B^\star\\A & \la C\end{smallmatrix}\right]$. The equation 
\[AXB+CX^\star D= E\]
has a unique solution, for any right-hand side $E$, if and only if $\cQ(\la)$ is regular and one of the following situations holds:

\begin{itemize} 
\item[{\rm(1)}] $p=m \neq n=q$, either $m < n$ and $A$ is invertible or 
  $m > n$ and $B$ is invertible, and 
\begin{itemize} 
\item If $\star=\top$, $\widehat\Lambda(\cQ)\setminus\{\pm1\}$ is reciprocal free and $m_1(\cQ)=m_{-1}(\cQ)\leq1$.
\item If $\star=*$, $\widehat\Lambda(\cQ)$ is $*$-reciprocal free.
\end{itemize}

\item[{\rm(2)}] $p=n\neq m=q$, either $m > n$ and $C$ is invertible
  or $m < n$ and $D$ is invertible, and
\begin{itemize} 
\item If $\star=\top$, $\widetilde\Lambda(\cQ)\setminus\{\pm1\}$ is reciprocal free and $m_1(\cQ)=m_{-1}(\cQ)\leq1$.
\item If $\star=*$, $\widetilde\Lambda(\cQ)$ is $*$-reciprocal free.
\end{itemize}

\item[{\rm (3)}] $p=m=n=q$, and 
\begin{itemize} 
\item If $\star=\top$, $\Lambda(\cQ)\setminus\{\pm1\}$ is reciprocal free and $m_1(\cQ)=m_{-1}(\cQ)\leq1$.
\item If $\star=*$, $\Lambda(\cQ)$ is $*$-reciprocal free.
\end{itemize}
\end{itemize}

\end{theorem}

\subsection{Vectorization}\label{size-sec}

Equation \eqref{gensylv} can be considered as a linear system in the entries of the unknown matrix $X$. The natural approach to get such system is applying the vectorization (vec) operator \cite[\S4.3]{hj-topics}. 

Set $X\in\CC^{m\times n}$, and let $A\in\CC^{p\times m},B\in\CC^{n\times q},C\in\CC^{p\times n},D\in\CC^{m\times q},E\in\CC^{p\times q}$. In the case $\star=\top$, after applying the vec operator we obtain a linear equation $M \opvec(X)=\opvec(E)$, with $M\in\CC^{(pq)\times (mn)}$ given by
\begin{equation}\label{m}
M=B^\top \otimes A+(D^\top\otimes C)\Pi,
\end{equation}
where $\Pi$ is a permutation matrix associated with the transposition~\cite[Equation~4.3.9b]{hj-topics}. 

In the case $\star=*$, some more care is needed, since the system obtained by vectorization is not linear over $\mathbb{C}$, due to the presence of conjugations. Nevertheless, we can separate real and imaginary parts as in \cite[\S1.1]{di16} and write it as a linear system over $\mathbb R$ of size $(2pq)\times(2mn)$ in $Y= \opvec(\onetwo {\re (X)}{\im (X)})$. 

The fact that Equation~\eqref{gensylv} is equivalent to a linear system has two important consequences. The first one is that~\eqref{gensylv} can have a unique solution, for any right-hand-side, only if the coefficient matrix of the linear system is square, that is, $mn=pq$. The second one is that, provided that $pq=mn$, the uniqueness of solution does not depend on the right-hand side: Equation \eqref{gensylv} has a unique solution for any $E$ if and only if the corresponding homogeneous equation
\begin{equation}\label{hom}
AXB+CX^\star D=0
\end{equation}
has only the trivial solution $X=0$. Hence, from now on, we assume $mn=pq$ and
we focus on Equation \eqref{hom} instead of Equation \eqref{gensylv}.

A different reformulation of the $\star=*$ case as a linear system (in the homogeneous case) is the following.
\begin{lemma}\label{lem:conjcase}
Equation \eqref{hom}, with $\star=*$, has a unique solution if and only if the linear system of equations
\begin{equation}\label{bigsystem}
\begin{split} 
AXB+CYD&=0,\\
D^*XC^*+B^*YA^*&=0,
\end{split}
\end{equation}
has a unique solution.
\end{lemma}
\begin{proof}
Let us first assume that \eqref{hom} has a nonzero solution $X$. Then this gives a nonzero solution $(X,X^*)$ of \eqref{bigsystem}.

To prove the converse, let $(X,Y)$ be a nonzero solution of \eqref{bigsystem}. Then, the matrix $X+Y^*$ is a solution of \eqref{hom}. If $X+Y^*$ is zero, then $Y=-X^*$, and in this case ${\mathfrak i}X$ is a nonzero solution of \eqref{hom}, with ${\mathfrak i}:=\sqrt{-1}$.
\end{proof}

The matrix associated to~\eqref{bigsystem} after applying the vec operator is
\begin{equation}\label{mstar}
M=\left[\begin{array}{cc}B^\top\otimes A&D^\top\otimes C\\\overline C\otimes D^*&\overline A\otimes B^*\end{array}\right].
\end{equation}

\section{Proof of Theorem \ref{thm:main}}\label{sec:starsylv}

Here we provide a proof of Theorem~\ref{thm:main}, which gives a complete characterization of the uniqueness of solution of \eqref{gensylv} for any right-hand side.

We split the proof in four cases:
\begin{itemize}
\item[(C1)] $mn\ne pq$;
\item[(C2)] $mn=pq$ and $p\not\in\{m,n\}$;
\item[(C3)] $mn=pq$ and $p\in\{m,n\}$, with $m\ne n$;
\item[(C4)] $mn=pq$ and $p\in\{m,n\}$, with $m=n$.
\end{itemize}

In Case (C4) all coefficients are square and Theorem~\ref{thm:main} holds, because it reduces to Theorem \ref{thm:di16}.

In Case (C1) the theorem is true because Equation \eqref{gensylv} fails to have a unique solution for any right-hand side $E$ and the conditions (1)--(3) cannot be fulfilled. Indeed, in order for \eqref{gensylv} to have a unique solution, for any right-hand side $E$, the coefficient matrix associated with Equation \eqref{gensylv} must be square, and this implies $mn=pq$, as explained in Section \ref{size-sec}. On the other hand, 
the conditions in each case (1)--(3) in the statement, imply $mn=pq$. 

Case (C2), in which every coefficient is non-square, namely $p\notin\{m,n\}$, will be treated in Section \ref{sec:rectangular}, and Case (C3), in which two coefficients are square and two are non-square, namely $p\in\{m,n\}$, with $m\ne n$, will be treated in Section \ref{sec:twosquare}.

\subsection{The case $mn=pq$ and $p\not\in\{m,n\}$}\label{sec:rectangular}

In this section we show that Theorem~\ref{thm:main} holds if $p\notin \{m,n\}$, with $mn=pq$. In particular we will show that, in this case, the pencil $\mathcal Q(\la)$ is singular and Equation \eqref{gensylv} fails to have a unique solution for the right hand side $E=0$. Note that, because of the restriction $mn=pq$, this also implies that $q\notin \{m,n\}$, so this situation covers all instances of Theorem \ref{thm:main} where none of the coefficient matrices are square.

We first show that $\mathcal Q(\la)$ is non-square and thus singular. Note that $\cQ(\la)$ has size $(p+q)\times(m+n)$. If $p+q=m+n$ this fact, together with the identity $mn=pq$, would imply $\{m,n\}=\{p,q\}$, since both $m,n$ and $p,q$ are the roots of the same quadratic polynomial, namely $x^2-(m+n)x+mn$.

Then, we show that  Equation \eqref{hom} admits a non-zero solution.
\begin{lemma}\label{lem:rectangular}
Let $A\in\CC^{p\times m},B\in\CC^{n\times q},C\in\CC^{p\times n},D\in\CC^{m\times q}$. If $mn=pq$ and $p\notin \{m,n\}$ then $AXB+CX^\star D=0$ has a nonzero solution.
\end{lemma}
\begin{proof} We consider separately four cases, depending on whether $p$ is smaller or larger than $m$ and $n$.

\begin{enumerate}

\item $p<\min\{m,n\}$. There are two nonzero vectors $u,v$ such that $Au=0$ and $Cv=0$, because of the dimensions of these two matrices. Then $X=uv^\star$ is a nonzero solution of \eqref{hom}.

\item If $p>\max\{m,n\}$, the identity $mn=pq$ implies $q<\min\{m,n\}$. Then, there are two nonzero vectors $u,v$ such that $v^\star B=0$, $u^\star D=0$, and $X=uv^\star$ is a nonzero solution of \eqref{hom}.

\item $m<p<n$. In this case, and because of the identity $mn=pq$, we have $m<q<n$ as well. Therefore, $m<\min\{p,q\}$. In particular, there exist nonzero vectors $u,v$ such that $u^\top A = 0$, $v^\top D^\top = 0$.

Now we consider the cases:
\begin{itemize}

\item[(a)] $\star=\top$. As argued in Section~\ref{size-sec}, Equation~\eqref{hom} is equivalent to the linear system $M\opvec{X} = 0$, with the matrix $M\in\mathbb{C}^{(mn)\times(mn)}$ as in~\eqref{m}. Then, $(v^\top \otimes u^\top)M = 0$, so $M$ is singular and \eqref{hom} has a nonzero solution.

\item[(b)] $\star=*$. As a consequence of Lemma~\ref{lem:conjcase}, Equation~\eqref{hom} has a nonzero solution if and only if the (square) matrix~\eqref{mstar} is singular. It is easy to verify that $\begin{bmatrix}
  v^\top \otimes u^\top & u^* \otimes v^*
\end{bmatrix} M = 0$, so $M$ is indeed singular.
\end{itemize}

\item $n<p<m$. By setting $Y=X^\star$, Equation \eqref{hom} is equivalent to $CYD+AY^\star B=0$, so we use the result for the previous case on this equation. \qedhere
\end{enumerate}
\end{proof}


\subsection{The case $mn=pq$ and $p\in\{m,n\}$, with $m\ne n$}\label{sec:twosquare}

In this section we show that Theorem~\ref{thm:main} holds for $p\in\{m,n\}$, with $mn=pq$ and $m\ne n$.

Since for $mn=pq$ the matrix associated with Equation \eqref{gensylv} is square, the unique solvability of the generalized $\star$-Sylvester equation \eqref{gensylv}, for any right-hand side, is equivalent to the existence of a unique solution of the homogeneous equation \eqref{hom}. Then, it is sufficient to prove the equivalence of conditions on the pencil $\mathcal Q(\la)$ with the uniqueness of solution of the homogeneous equation \eqref{hom}.


We have either $p=m$ or $p=n$, which
imply $q = n$ and $q = m$, respectively, due to the constraint
$mn = pq$. We will prove first the case in which $p=m>n=q$; then we will reduce the remaining cases: $p=m<n=q$, and $p=n\ne m$, to this one.

Let us assume that $p = m > n = q$, so that $D$ has more rows than columns, and there is some $u\neq0$ such that $u^\star D = 0$. If $B$ is singular, then there is some $v\neq0$ such that $v^\star B= 0$. Therefore $X = uv^\star$ is a nontrivial solution
of \eqref{hom}. 

Assume now that \eqref{hom} has a unique solution. Then, $B$ is
guaranteed to be nonsingular, and 
$AXB + CX^\star D = 0$ has a unique solution if and only if
$AX + CX^\star DB^{-1} = 0$ has a unique solution. 
Moreover, we can find an invertible matrix $Q\in\C^{m
\times m}$ such that
\begin{equation}\label{qd}
QDB^{-1}=\left[\begin{array}{c}D_1\\0\end{array}\right],
\end{equation}
with $D_1\in\CC^{n\times n}$. This allows us to rewrite
\eqref{hom}, after multiplying on the right by $B^{-1}$, and setting $Y=Q^{-\star}X$, in the equivalent form
\begin{equation}\label{eqb}
AQ^\star \left[\begin{array}{c}Y_1\\Y_2\end{array}\right]+C\left[\begin{array}{cc}Y_1^\star&Y_2^\star\end{array}\right] \left[\begin{array}{c}D_1\\0\end{array}\right]=0,
\end{equation}
where $Y_1$ has size $n\times n$ and $Y_2$ has size $(m-n)\times n$. If we partition $AQ^\star$ conformally as
\begin{equation}\label{aqstar}
AQ^\star=\left[\begin{array}{cc}\widetilde A_{11}&\widetilde A_{12}\\\widetilde A_{21}&\widetilde A_{22}\end{array}\right],
\end{equation}
with $\widetilde A_{11}\in\CC^{n\times n},\widetilde A_{22}\in\CC^{(m-n)\times(m-n)}$, then the block $\left[\begin{smallmatrix}\widetilde A_{12}\\\widetilde A_{22}\end{smallmatrix}\right]$ has full column rank. 
If that was not the case we could find $Y_2 \neq 0$ such that
\[
  \begin{bmatrix}
    \widetilde A_{12} \\
    \widetilde A_{22} \\
  \end{bmatrix} Y_2 = 0,
  \]
  and this would imply
  \[
  AQ^\star \begin{bmatrix}
    0 \\
    Y_2 \\
  \end{bmatrix} + C \begin{bmatrix}
    0 & Y_2^\star
  \end{bmatrix} \begin{bmatrix}
    D_1 \\ 
    0   \\
  \end{bmatrix} = 0, 
\]
so equation~\eqref{eqb} would have a nontrivial solution. 
Then, there is an invertible matrix $U\in\C^{m\times m}$ such that
\[
UAQ^\star=U\left[\begin{array}{cc}\widetilde A_{11}&\widetilde A_{12}\\\widetilde A_{21}&\widetilde A_{22}\end{array}\right]=
\left[\begin{array}{cc}\widehat A_{11}&0\\\widehat A_{21}&\widehat A_{22}\end{array}\right],
\]
with $\widehat A_{22}\in\CC^{(m-n)\times (m-n)}$ nonsingular. If we set $UC=\left[\begin{smallmatrix}\widehat C_1\\\widehat C_2\end{smallmatrix}\right]$, with $\widehat C_1\in\CC^{n\times n},\widehat C_2\in\CC^{(m-n)\times n}$ then, after multiplying on the left by $U$, \eqref{eqb} is equivalent to the system
\[
\begin{array}{ccc}
\widehat A_{11} Y_1+\widehat C_1Y_1^\star D_1&=&0,\\
\widehat A_{22} Y_2&=&-(\widehat A_{21}Y_1+\widehat C_2Y_1^\star D_1).
\end{array}
\]
Since $\widehat A_{22}$ is nonsingular, the above system has a unique
solution if and only if the first equation
\begin{equation}\label{redsylv}
\widehat A_{11}Y_1+\widehat C_1Y_1^\star D_1=0
\end{equation}
has a unique solution. 

We are now ready to relate the uniqueness of solution of \eqref{hom} to the
spectral properties of the pencil $\cQ(\la)$ in the statement of the theorem. 
We perform the following left and right invertible transformations
to $\cQ(\la)$:
\begin{equation}\label{product-q}
\left[\begin{array}{cc}B^{-\star}&0\\0&U\end{array}\right]
\left[\begin{array}{cc}\la D^\star&B^\star\\A&\la C\end{array}\right]
\left[\begin{array}{cc}Q^\star&0\\0&I\end{array}\right]=
\left[\begin{array}{ccc}\la D_1^\star&0&I\\\widehat A_{11}&0&\la\widehat C_1\\\widehat A_{21}&\widehat A_{22}&\la \widehat C_2\end{array}\right].
\end{equation}
Set
\[
\widehat\cQ(\la)=\left[\begin{array}{cc}\la D_1^\star&I\\\widehat A_{11}&\la\widehat C_1\end{array}\right].
\]
Then, by \eqref{product-q},
$
\det \cQ(\la)= \alpha \det \widehat\cQ(\la),
$
where
$\alpha=\pm(\det \widehat A_{22}\det B^\star)/(\det U\det Q^\star)$ is a nonzero constant.

Since all coefficient matrices in \eqref{redsylv} are square and have the same size, namely $n\times n$, Theorem \ref{thm:di16} implies that $\widehat \cQ(\la)$ is regular (so $\cQ(\la)$ is regular as well) and 
\begin{itemize}
\item If $\star=\top$, $\Lambda(\widehat\cQ)\setminus\{\pm1\}$ is reciprocal free and $m_1(\widehat\cQ)=m_{-1}(\widehat\cQ)\leq1$.

\item If $\star=*$, $\Lambda(\widehat\cQ)$ is $*$-reciprocal free.
\end{itemize}

Note that \eqref{product-q} implies that $\widehat \Lambda(\cQ)=\Lambda(\widehat\cQ)$, since $\widehat Q$ is obtained by deflating $m-n$ infinite eigenvalues from the pencil in the right hand side of \eqref{product-q}. So the previous two conditions are equivalent to the conditions on the spectrum of $\cQ(\la)$ in the statement of the theorem.

To prove the converse, let us assume that $B$ is invertible, and that $\cQ(\la)$ is regular and its spectrum satisfies the conditions in the statement of the theorem. Then we can define the matrix $Q$ as in \eqref{qd} and we arrive at \eqref{aqstar}. Again, the block $\left[\begin{smallmatrix}\widetilde A_{12}\\\widetilde A_{22}\end{smallmatrix}\right]$ has full column rank since, otherwise, the pencil $\cQ(\la)$ would be singular. This is an immediate consequence of the identity:
\[
\left[\begin{array}{cc}B^{-\star}&0\\0&I\end{array}\right]
\left[\begin{array}{cc}\la D^\star&B^\star\\A&\la C\end{array}\right]
\left[\begin{array}{cc}Q^\star&0\\0&I\end{array}\right]=
\left[\begin{array}{ccc}\la D_1^\star&0&I\\\widetilde A_{11}&\widetilde A_{12}&\la \wt C_1\\\widetilde A_{21}&\widetilde A_{22}&\la \wt C_2\end{array}\right],
\]
where $\left[\begin{smallmatrix}{\wt C_1}\\{\wt C_2}\end{smallmatrix}\right]=C$.

Proceeding as before, we conclude, that \eqref{hom} is equivalent to \eqref{redsylv}. Using the fact that $\widehat \Lambda(\cQ)=\Lambda(\widehat\cQ)$ and applying Theorem \ref{thm:di16}, with the hypotheses on $\cQ(\la)$, to Equation \eqref{redsylv}, we conclude that the latter has a unique solution, and this implies that \eqref{hom} has a unique solution.

Now that we have proved the case $p=m>n=q$, we consider the case $p=m<n=q$ and then the case $p=n\ne m$. 

Let us assume that $p=m<n=q$. After applying the $\star$ operator in \eqref{hom} and setting $Y=X^\star$, we arrive at the equivalent equation
$
B^\star YA^\star+D^\star Y^\star C^\star=0.
$
This equation is of the form \eqref{hom}, with the coefficients of the first summand being square, $B^\star\in\CC^{n\times n},A^\star\in\CC^{m\times m}$ and $n>m$, so we are in the same conditions  as before. 
Applying the result just proved for this case, we get that the unique solvability is equivalent to requiring that $A$ is invertible and that the pencil 
\[
\widetilde\cQ(\la)=\left[\begin{array}{cc}\la C&A\\B^\star&\la D^\star\end{array}\right]
\]
satisfies the conditions in the statement of the theorem. But, since
\[
\widetilde\cQ(\la)=\left[\begin{array}{cc}0&I\\I&0\end{array}\right]\cQ(\la)\left[\begin{array}{cc}0&I\\I&0\end{array}\right],
\]
this is equivalent to requiring that $\cQ(\la)$ satisfies these conditions as well.

For the case $p=n\neq m$, we apply the $\star$ operator in \eqref{hom}, and we arrive at the equivalent equation
$
D^\star XC^\star+B^\star X^\star A^\star=0,
$
whose coefficients are in the conditions of the previous case. The pencil associated to this last equation is
\[
\wt\cQ(\la)=\left[\begin{array}{cc}\la A&C\\D^\star&\la B^\star\end{array}\right].
\]
This pencil is the reversal of the pencil:
\[
\left[\begin{array}{cc}0&I\\I&0\end{array}\right]\cQ(\la),
\]
so $\Lambda(\wt\cQ)=\Lambda^{-1}(\cQ):=\{\la^{-1}:\la\in\Lambda(\cQ)\}$, including multiplicities. In particular, the conditions on being ($*$-)reciprocal free in the statement are the same for both pencils, and the roles of the zero and the infinite eigenvalue are exchanged.



\begin{remark}
The conditions $n=q$ in part 1, and $m=q$ in part 2 in Theorem \ref{thm:main} are redundant, but we have included them for emphasis. These conditions are a consequence of the fact that $\cQ(\la)$ in \eqref{pencil} is regular and the other conditions on the size, namely $p=m$ and $p=n$, respectively. As indicated in the proof of Theorem \ref{thm:main}, since $\cQ(\la)$ has size $(p+q)\times(m+n)$, if it is regular, it must be, in particular, square, and this implies $m+n=p+q$. 
\end{remark}

\subsection{Necessity of the invertibility conditions} \label{sec:counterexample}

The characterization of the uniqueness of solution of \eqref{gensylv} in Theorem~\ref{thm:main} involves, in cases 1 and 2, the invertibility of some of the coefficient matrices. One might wonder if these conditions are
really needed, or whether they could be stated in terms of spectral properties of the pencil $\cQ(\la)$. However, the following example shows that the uniqueness of solution does not depend solely on the eigenvalues of
$\cQ(\la)$. Consider the following
generalized $\top$-Sylvester equations (the same example
works for the $\star = *$ case): 
\begin{align}
  \begin{bmatrix}
    1 &0\\
    0& 1 \\
  \end{bmatrix} \begin{bmatrix}
    x \\ y 
  \end{bmatrix} \begin{bmatrix} 0 \end{bmatrix} + 
  \begin{bmatrix}
    1 \\
    0 \\
  \end{bmatrix} \begin{bmatrix}
    x & y \\
  \end{bmatrix} \begin{bmatrix}
    1 \\
    0
  \end{bmatrix} = 0,\label{eq1} \\
  \begin{bmatrix}
    0 & 0 \\
    0 & 1 \\
  \end{bmatrix} \begin{bmatrix}
  x \\ y 
  \end{bmatrix} \begin{bmatrix} 1 \end{bmatrix} + 
  \begin{bmatrix}
  1 \\
  0 \\
  \end{bmatrix} \begin{bmatrix}
  x & y \\
  \end{bmatrix} \begin{bmatrix}
  1 \\
  0
  \end{bmatrix} = 0\label{eq2}.
\end{align}
The above equations have associated pencils defined as follows: 
\[
  \cQ_1(\la) = \left[\begin{array}{cc|c}
    \la & 0 & 0 \\\hline
    1   & 0 & \la \\
    0   & 1 & 0 \\ 
  \end{array}\right], \qquad 
    \cQ_2(\la) = \left[\begin{array}{cc|c}
    \la & 0 & 1 \\\hline
    0   & 0 & \la \\
    0   & 1 & 0 \\ 
    \end{array}\right].
\]
The above pencils are the same up to row and column permutations, 
so they have not just the same eigenvalues, but also the same Kronecker canonical form.
However, the corresponding generalized Sylvester equations \eqref{eq1}--\eqref{eq2} can be rewritten, 
respectively, as
\[
  x = 0, \qquad \mbox{and}\qquad
  x=y=0.
\]
Then \eqref{eq1} has infinitely many solutions, while 
\eqref{eq2} has a unique solution. 

\section{Some corollaries}

The characterization given in Theorem \ref{thm:main} depends on spectral properties of the pencil $\cQ(\la)$ in \eqref{pencil}, which has twice the size of the coefficient matrices of Equation \eqref{gensylv}. With some additional effort, we can provide a characterization in terms of pencils with exactly the same size.

\begin{corollary}\label{thm:nxn}
Let $A\in\CC^{p\times m},B\in\CC^{n\times q},C\in\CC^{p\times n},$ and $D\in\CC^{m\times q}$. Then the equation 
$AXB+CX^\star D= E$
has a unique solution, for any right-hand side $E$, if and only if one of the following situations holds:

\begin{itemize} 
\item[{\rm(a)}] $p=m \leq n=q$, $A$ is invertible, the pencil $\cP_1(\la):= B^\star-\la  D^\star A^{-1}C$ is regular and
\begin{itemize} 
\item If $\star=\top$, $\widehat{\Lambda}(\cP_1)\setminus\{1\}$ is reciprocal free and $m_1(\cP_1)\leq1$.
\item If $\star=*$, $\widehat{\Lambda}(\cP_1)$ is $*$-reciprocal free.
\end{itemize}

\item[{\rm(b)}] $p=m\geq n=q$, $B$ is invertible, the pencil $\cP_2(\la):= A^\star-\la DB^{-1}C^\star$ is regular and
\begin{itemize} 
\item If $\star=\top$, $\widehat{\Lambda}(\cP_2)\setminus\{1\}$ is reciprocal free and $m_1(\cP_2)\leq1$.
\item If $\star=*$, $\widehat{\Lambda}(\cP_2)$ is $*$-reciprocal free.
\end{itemize}

\item[{\rm (c)}] $p=n\leq m=q$, $C$ is invertible, the pencil $\cP_3(\la):= D^\star-\la B^\star C^{-1}A$ is regular and
\begin{itemize} 
\item If $\star=\top$, $\widehat{\Lambda}(\cP_3)\setminus\{1\}$ is reciprocal free and $m_1(\cP_3)\leq1$.
\item If $\star=*$, $\widehat{\Lambda}(\cP_3)$ is $*$-reciprocal free.
\end{itemize}

\item[{\rm(d)}] $p=n\geq m=q$, $D$ is invertible, the pencil $\cP_4(\la):=C^\star-\la BD^{-1}A^\star$ is regular and
\begin{itemize} 
\item If $\star=\top$, $\widehat{\Lambda}(\cP_4)\setminus\{1\}$ is reciprocal free and $m_1(\cP_4)\leq1$.
\item If $\star=*$, $\widehat{\Lambda}(\cP_4)$ is $*$-reciprocal free.
\end{itemize}

\end{itemize}
\end{corollary}
\begin{proof} Let us assume first that \eqref{gensylv} has a unique solution, for any right-hand side $E$. Then \cite[Th. 3]{dipr} implies that at least one of the following situations holds: (C1) $p=m< n=q$ and $A$ is invertible, (C2) $p=m>n=q$ and $B$ is invertible, (C3) $p=n<m=q$ and $C$ is invertible, (C4) $p=n>m=q$ and $D$ is invertible, or (C5) $p=m=n=q$. Let us first assume that case (C1) holds. We can perform the following unimodular equivalence on $\cQ(\la)$:
  \begin{equation}\label{unimodular}
    \begin{bmatrix}
      I & -\la D^{\star} A^{-1} \\
      0& I \\
    \end{bmatrix} \begin{bmatrix}
      \la D^\star & B^\star \\
      A & \la C \\
    \end{bmatrix} = \begin{bmatrix}
      0 & B^\star - \la^2 D^{\star} A^{-1} C \\
      A & \la C \\
    \end{bmatrix}.
  \end{equation}
	Taking determinants in \eqref{unimodular} we arrive at
	\begin{equation}\label{deter}
	\det (\cQ(\la)) = \pm\det (A)\det(\cP_1(\la^2)).
	\end{equation}
	This shows that $\cP_1$ is regular.
	Note that $D^\star A^{-1}C$ has rank at most $m<n$, hence $\det(\cP_1(\la))$ has degree at most $m$ and $\abs{n-m}$ dimension-induced infinite eigenvalues are present in $\Lambda(\cP_1)$. Similarly, $\cQ(\la)$ has $\abs{n-m}$ dimension-induced infinite eigenvalues. The left- and right-hand sides of Equation~\eqref{deter} are nonzero polynomials in $\la$ with degree at most $2m$; therefore we have $\widehat\Lambda(\cQ)=\sqrt{\widehat{\Lambda}(\cP_1)}:=\bigl\{\mu:\ \mu^2\in\widehat{\Lambda}(\cP_1)\bigr\}$, including multiplicities and core infinite eigenvalues. Then \cite[Th. 3]{dipr} implies that part (a) in the statement holds. 

If case (C4) holds, then we apply the $\star$ operator in \eqref{gensylv} and the previous arguments to the new equation and its corresponding pencil $C-\la AD^{-\star}B^\star$, namely $\bigl(\mathcal P_4(\la^\star)\bigr)^\star$, and part (d) of the statement follows.

 If case (C3) holds, then after introducing the change of variables $Y=X^\star$, the roles of $A,B$ and $C,D$ are exchanged, so we apply the same arguments as in case (C1) to the corresponding pencil, $\cP_3(\la)$ and we get part (c).

In case (C2), we apply the $\star$ operator in \eqref{gensylv} and introduce the change of variables $Y=X^\star$. Then we apply the same arguments as for case (C1) to the new equation and its corresponding pencil $A-\la CB^{-\star}
D^\star$, namely $\bigl(\mathcal P_2( \la^\star)\bigr)^\star$, and part (b) of the statement follows.

Finally, if we are in case (C5), \cite[Cor. 12]{di16} guarantees that at least one of $A,B,C,D$ is invertible and thus at least one of (a)--(d) in the statement holds, and we are done.

To prove the converse, let us assume that any of (a)--(d) in the statement holds. Then, reversing the previous arguments and using \eqref{eq:rev}, we can conclude that at least one of the situations (i)--(iii) in the statement of \cite[Th. 3]{dipr} occurs, and \cite[Th. 3]{dipr} implies that \eqref{gensylv} has a unique solution, for any right-hand side.

\end{proof}

As another consequence of Theorem \ref{thm:main} we get an extension of \cite[Lemma 5.10]{bk06} and \cite[Lemma 8]{ksw09} for the $\star$-Sylvester equation $AX+X^\star D=E$ (see also \cite[Th. 10, Th. 11]{ddgmr11}) to the case of rectangular coefficients, showing that when the coefficients are non-square the equation cannot be uniquely solvable.

\begin{corollary}\label{sylv.coro}
Let $A\in\CC^{n\times m}$ and $D\in\CC^{m\times n}$. Then the equation $AX+X^\star D=E$ has a unique solution, for any right-hand side $E$, if and only if the matrix pencil $\cP(\la)=A-\la D^\star$ is regular and:
\begin{itemize}
\item If $\star=\top$, $\Lambda(\cP)\setminus\{1\}$ is reciprocal free and $m_{1}(\cP)\leq1$.
\item If $\star=*$, $\Lambda(\cP)$ is $*$-reciprocal free.
\end{itemize}
In particular, if the coefficients are non-square, then the equation cannot have a unique solution, for any right-hand side $E$.
\end{corollary}

A more interesting situation involves the $\star$-Stein equation $AXB+X^\star=E$, for which we get the following result that generalizes Theorem 10 of \cite{di16} (see also the references in \cite{di16}).
\begin{corollary}\label{stein.coro}
Let $A,B\in\CC^{n\times m}$. Then the equation $AXB+X^\star=E$ has a unique solution, for any right-hand side $E$, if and only if the following conditions hold:
\begin{itemize}
\item If $\star=\top$, $\Lambda(AB^\top)\setminus\{1\}$ is reciprocal free and $m_{1}(AB^\top)\leq1$.
\item If $\star=*$, $\Lambda(AB^*)$ is $*$-reciprocal free.
\end{itemize}
\end{corollary}
\begin{proof}
It is sufficient to observe that the condition in Corollary \ref{thm:nxn} (taking $C=I$, $D=I$) is equivalent to the condition stated on the spectrum of $AB^\star$, for each of the cases in Corollary \ref{thm:nxn}. If $m>n$, we are in case (c), with $\mathcal P_3=I-\la B^\star A$. The eigenvalues of $\mathcal P_3$ are the reciprocals of the eigenvalues of $B^\star A$. Note that $B^\star A$ has $m-n$ dimension-induced zero eigenvalues and $\widetilde{\Lambda}(B^\star A) = \Lambda(AB^\star)$ (this equality follows from~\cite[Theorem~1.3.20]{hj}). Hence the set $\Lambda(AB^\star)$ is the reciprocal of $\widehat{\Lambda}(\cP_3)$, so one of the two is ($*$-)reciprocal-free if and only if the other is, while the multiplicity of $1$ is the same in both spectra. 

Similarly, if $m < n$, we can take $\cP_4(\lambda)=I-\la BA^\star$; then $\widehat{\Lambda}(\cP_4)$ is the reciprocal of $\widetilde{\Lambda}(BA^\star)$, or, applying the $\star$ operator, the ($*$-)reciprocal of $\widetilde{\Lambda}(AB^\star)$. Since a matrix never has $\infty$ as an eigenvalue, $\widetilde{\Lambda}(AB^\star)$ is a ($*$-)reciprocal-free set if and only if $\Lambda(AB^\star)$ is so, regardless of the additional zero eigenvalues.

The cases with $m=n$ can be proved in a similar way.
\end{proof}

Comparing Corollaries \ref{sylv.coro} and \ref{stein.coro}, we see an interesting difference between the unique solvability of the $\star$-Sylvester equation $AX+X^\star D=E$ and the $\star$-Stein equation $AXB+X^\star=E$. In the first case, the equation can not have a unique solution, for any right-hand side $E$, unless the coefficient matrices are square. However, the $\star$-Stein equation can have unique solution, for any right-hand side $E$, for rectangular coefficient matrices $A,B\in\CC^{n\times m}$. As an elementary example, consider the matrices $A=\left[\begin{array}{cc}1&1\end{array}\right]$ and $B=\left[\begin{array}{cc}1&2\end{array}\right],$ with $n=1,m=2$. The matrix $AB^\top=3$ satisfies the conditions in the statement of Corollary \ref{stein.coro} (for both $\star=\top,*$), so $AXB+X^\star=E$ has a unique solution, for any right-hand side $E$. Indeed, the equation, in the case $\star=\top$, is
$$
\left[\begin{array}{cc}1&1\end{array}\right]\left[\begin{array}{c}x_1\\x_2\end{array}\right]\left[\begin{array}{cc}1&2\end{array}\right]+
\left[\begin{array}{cc}x_1&x_2\end{array}\right]=\left[\begin{array}{cc}e_1&e_2\end{array}\right],
$$
which has a unique solution for any $e_1,e_2\in\CC$. More precisely, this solution is given by
$$
x_1=\frac{3e_1-e_2}{4},\qquad x_2=\frac{e_2-e_1}{2}.
$$
We leave to the reader to check that the system also has a unique solution, for any $e_1,e_2\in\CC$, in the case $\star=*$.

\section{Explicit characterization for the generalized Sylvester equation} \label{sec:kosir_criteria}

In this section, we provide an explicit solution of problems {\bf SR}, {\bf OR}, and {\bf UR} for Equation~\eqref{eq:gensylvr}. These characterizations follow from the results and lemmas in~\cite{kosir-tech}, but we state them explicitly in Theorem \ref{thm:kosir_characterization}. 
Unlike the characterization given in Theorem \ref{thm:main} of {\bf UR} for Equation~\eqref{gensylv}, the characterizations for Equation \eqref{eq:gensylvr} depend on further constraints on the Kronecker canonical form (KCF) of the matrix pencils $A-\la C$ and $D^\top-\la B^\top$, and not just on their spectrum. Though the KCF is a standard canonical form that can be found in most of the basic references on matrix pencils, we refer the reader to \cite[Th. 2]{ddgmr11}, since we follow the notation in that paper. In particular, $J(\alpha)$ denotes a {\em Jordan block} associated with the eigenvalue $\alpha$, including $\alpha=\infty$ (which is denoted by $N$ in \cite{ddgmr11}), $L_\varepsilon$ denotes a {\em right singular block} of size $\varepsilon\times (\varepsilon+1)$, and $L_\eta^\top$ denotes a {\em left singular block} of size $(\eta+1)\times \eta$. We denote by $\mathbb Z^+$ the set of positive integers.

\begin{theorem} \label{thm:kosir_characterization} Let $A,C\in\CC^{p\times m}$ and $B,D\in\CC^{n\times q}$. 
\begin{description}
	\item[SR]
	Equation~\eqref{eq:gensylvr} has \emph{at least} one solution, for any right-hand side $E$, if and only if the following two conditions are satisfied:
	\begin{itemize}
		\item the (possibly singular) pencils $A-\lambda C$ and $D^\top - \lambda B^\top$ have no common eigenvalues, \emph{and}
		\item if the KCF of either $A-\la C$ or $D^\top-\la B^\top$ contains a block $L_\eta^\top$, then the KCF of the other pencil is a direct sum of blocks $L_{\varepsilon_i}$ with $\varepsilon_i\leq \eta$.  
	\end{itemize}
	\item[OR] Equation~\eqref{eq:gensylvr} has \emph{at most} one solution, for any right-hand side $E$, if and only if the following two conditions are satisfied:
	\begin{itemize}
		\item the (possibly singular) pencils $A-\lambda C$ and $D^\top - \lambda B^\top$ have no common eigenvalues, \emph{and}
		\item if the KCF of either $A-\la C$ or $D^\top-\la B^\top$ contains a block $L_\varepsilon$, then the KCF of the other pencil is a direct sum of blocks $L_{\eta_i}^\top$ with $\eta_i\leq \varepsilon$.  
	\end{itemize}
	\item[UR] Equation~\eqref{eq:gensylvr} has \emph{exactly} one solution, for any right-hand side $E$, if and only if one of the following situations hold:
	\begin{itemize}
		\item the pencils $A-\lambda C$ and $D^\top - \lambda B^\top$ are regular and have no common eigenvalues, \emph{or}
		\item there is some $s \in \mathbb{Z}^+$ such that the KCF of either $A-\la C$ or $B^\top - \la D^\top$ is a direct sum of blocks $L_s$, and the KCF of the other pencil is a direct sum of blocks $L_s^\top$.
	\end{itemize}
\end{description}
\end{theorem}
\begin{proof}
Let $P_1(A-\lambda C)Q_1=\widehat{A} - \lambda\widehat{C}$ and $P_2^\top(D^\top-\lambda B^\top)Q_2^\top=\widehat{D}^\top - \lambda\widehat{B}^\top$ be the KCFs of $A-\la C$ and $D^\top-\la B^\top$, respectively. Equation \eqref{eq:gensylvr} is equivalent to
\begin{equation} \label{hatsylv}
	\widehat{A}\widehat{X}\widehat{B} - \widehat{C}\widehat{X}\widehat{D} = \widehat{E}, \quad \text{with }\widehat{X} = Q_1^{-1}XQ_2^{-1}, \quad \widehat{E} = P_1EP_2.	
\end{equation}
Partitioning the matrices conformably with the (possibly rectangular) blocks in the KCFs, we get
\begin{align*}
\begin{bmatrix}
  \widehat{A}_{11}\\
  & \widehat{A}_{22}\\
  && \ddots \\
  &&& \widehat{A}_{pp}
\end{bmatrix}
\begin{bmatrix}
  \widehat{X}_{11} &\widehat{X}_{12} & \dots & \widehat{X}_{1q}\\
  \widehat{X}_{21} &\widehat{X}_{22} & \dots & \widehat{X}_{2q}\\
  \vdots & \vdots & \ddots & \vdots \\
  \widehat{X}_{p1} &\widehat{X}_{p2} & \dots & \widehat{X}_{pq}\\  
\end{bmatrix}
\begin{bmatrix}
  \widehat{B}_{11}\\
  & \widehat{B}_{22}\\
  && \ddots \\
  &&& \widehat{B}_{qq}
\end{bmatrix}\\
-
\begin{bmatrix}
  \widehat{C}_{11}\\
  & \widehat{C}_{22}\\
  && \ddots \\
  &&& \widehat{C}_{pp}
\end{bmatrix}
\begin{bmatrix}
  \widehat{X}_{11} &\widehat{X}_{12} & \dots & \widehat{X}_{1q}\\
  \widehat{X}_{21} &\widehat{X}_{22} & \dots & \widehat{X}_{2q}\\
  \vdots & \vdots & \ddots & \vdots \\
  \widehat{X}_{p1} &\widehat{X}_{p2} & \dots & \widehat{X}_{pq}\\  
\end{bmatrix}
\begin{bmatrix}
  \widehat{D}_{11}\\
  & \widehat{D}_{22}\\
  && \ddots \\
  &&& \widehat{D}_{qq}
\end{bmatrix}\\
=
\begin{bmatrix}
  \widehat{E}_{11} &\widehat{E}_{12} & \dots & \widehat{E}_{1q}\\
  \widehat{E}_{21} &\widehat{E}_{22} & \dots & \widehat{E}_{2q}\\
  \vdots & \vdots & \ddots & \vdots \\
  \widehat{E}_{p1} &\widehat{E}_{p2} & \dots & \widehat{E}_{pq}\\  
\end{bmatrix},
\end{align*}
and~\eqref{hatsylv} is equivalent to the system of $pq$ independent equations
\begin{equation} \label{smallsystems}
	\widehat{A}_{ii}\widehat{X}_{ij}\widehat{B}_{jj} - \widehat{C}_{ii}\widehat{X}_{ij}\widehat{D}_{jj} = \widehat{E}_{ij}, \quad i=1,2,\dots,p, \, j=1,2,\dots,q.	
\end{equation}
In particular, the existence (resp. uniqueness) of solution $X$ of \eqref{eq:gensylvr}, for any right-hand side $E$, is equivalent to the simultaneous existence (resp. uniqueness) of the solution $\widehat{X}_{ij}$ of \eqref{smallsystems}, for any right-hand side $\widehat{E}_{ij}$, and for any $i=1,2,\hdots,p,j=1,2,\hdots,q$.

Solvability and uniqueness conditions for the systems~\eqref{smallsystems}, where $\widehat{A}_{ii}-\lambda \widehat{C}_{ii}$ and $\widehat{D}_{jj}^\top-\lambda \widehat{B}_{jj}^\top$ are blocks from the KCF, are provided in~\cite{kosir-tech}. We recall them in Table~\ref{tbl:kronblocks}.
\begin{table}
\begin{tabular}{ccp{3.5cm}p{3.5cm}}
\toprule
$A-\lambda C$ & $D^\top - \lambda B^\top$ & At least one solution? & At most one solution?\\
\midrule
$J(\alpha)$ & $J(\alpha)$ & No & No\\
$J(\alpha)$ & $J(\beta)$  ($\beta\neq \alpha$) & Yes & Yes\\
$L_\eta^\top$& $L_\varepsilon$ & Only when $\eta\geq \varepsilon$ &Only when $\eta\leq \varepsilon$\\
$L_\varepsilon$& $L_\eta^\top$ & Only when $\eta\geq \varepsilon$ & Only when $\eta\leq \varepsilon$\\
$L_{\eta_1}^\top$ & $L_{\eta_2}^\top$ & No& Yes\\
$L_{\varepsilon_1}$ &$L_{\varepsilon_2}$& Yes & No\\
$J(\alpha)$&$L_\varepsilon$  & Yes & No \\
$J(\alpha)$ & $L_\eta^\top$  & No & Yes \\
$L_\eta^\top$  &$J(\alpha)$& No & Yes\\
$L_\varepsilon$ & $J(\alpha)$  & Yes & No \\
\bottomrule
\end{tabular}
\caption{\footnotesize Summary of the results on existence and uniqueness of solutions for the equation $AXB-CXD=E$ when $A-\lambda C$ and $D^\top - \lambda B^\top$ are Kronecker blocks, obtained from the lemmas in~\cite[Sec.~4 and~5]{kosir-tech}. The Jordan blocks include $\alpha=\infty$.
} \label{tbl:kronblocks}
\end{table}
One can check, using this table, that each equation of the system~\eqref{smallsystems} has at least one solution if and only if the \textbf{SR} conditions in the statement of the theorem hold; similarly, each equation has at most one solution if and only if the~\textbf{OR} conditions hold, and it has exactly one solution if and only if the~\textbf{UR} conditions hold.
\end{proof}

Theorem \ref{thm:kosir_characterization} shows that the characterization of \textbf{SR}, \textbf{OR}, and \textbf{UR} for the generalized Sylvester equation \eqref{eq:gensylvr} depends on the KCF of the pencils $A-\la C$ and $D^\top-\la B^\top$. It is natural to expect a characterization of \textbf{UR} for the generalized $\star$-Sylvester equation \eqref{gensylv} depending also on the KCF of the pencil ${\cal Q}(\la)$ in Theorem \ref{thm:main}. However, the example provided in Section \ref{sec:counterexample} shows that this is not the case. This example presents two different equations with different behavior (Equation \eqref{eq2} has a unique solution, for any right-hand side, whereas Equation \eqref{eq1} has not). However, the associated pencils, ${\cal Q}_2(\la)$ and ${\cal Q}_1(\la)$, respectively, have the same KCF. 

\subsection{An alternative characterization of UR}

The criterion for unique solvability for each right-hand side (\textbf{UR}) for the generalized Sylvester equation~\eqref{eq:gensylvr} gives a peculiar restriction that can be expressed in terms of the ratios between the two dimensions of the involved pencils.

\begin{corollary}\label{cor_s}
Let $A,C\in\mathbb{C}^{p\times m}$ and $B,D\in\mathbb{C}^{n\times q}$. If the generalized Sylvester equation~\eqref{eq:gensylvr} has a unique solution for any right-hand side, then $p/m=n/q=d$, for some
\[
d \in \{1 \} \cup \left\{\frac{s}{s+1} : s \in\mathbb{Z}^+\right\} \cup \left\{\frac{s+1}{s} : s \in\mathbb{Z}^+\right\}.
\]
\end{corollary}
\begin{proof}
By Theorem \ref{thm:kosir_characterization}, there are two possibilities in order for \eqref{eq:gensylvr} to have a unique solution, for any right-hand side: either the pencils $A-\la C$ and $D^\top-\la B^\top$ are regular with no common eigenvalues or one of the two pencils is a direct sum of blocks $L_s$ and the other is a direct sum of blocks $L_s^\top$. In the former case, $A$ and $B$ must be square ($p=m$ and $n=q$) and then $d=1$ in the statement. In the latter case, if $A-\la C$ is a direct sum of $k$ right singular blocks $L_s$ and $D^\top-\la B^\top$ is a direct sum of $\ell$ left singular blocks $L_s^\top$, then $A$ has size $(ks)\times (k(s+1))$, while $B$ has size $(\ell s)\times (\ell(s+1))$, and then $d=p/m=n/q=s/(s+1)$, while if $A-\la C$ is a direct sum of $k$ left singular blocks $L_s^\top$ and $D^\top-\la B^\top$ is a direct sum of $\ell$ right singular blocks $L_s^\top$, we have $d=(s+1)/s$.
\end{proof}

The necessary condition given in Corollary \ref{cor_s} is not sufficient. In order to give a sufficient condition, we need the following result. 

\begin{lemma}\label{lem_acmatrix}
Let $A,C\in\mathbb{R}^{p\times m}$, and $\varepsilon\in\mathbb{Z}^+$. Then, the KCF of the pencil $A-\lambda C$ consists only of blocks of the form 
$L_\varepsilon$ if and only if the $p(\varepsilon+1) \times m\varepsilon$  matrix
\[
M_\varepsilon(A,C) :=
	\begin{bmatrix}
	  A \\
	  C & A \\
	  & C & \ddots\\
	  & & \ddots & A\\
	  & & & C
	\end{bmatrix}
\] is square and invertible. Similarly, the KCF of $A-\lambda C$ is composed uniquely of blocks of the form $L_\eta^\top$ if and only if $M_\eta(A^\top,C^\top)$ is square and invertible.
\end{lemma}
\begin{proof}
Let us first suppose that $M_\varepsilon(A,C)$ is square and invertible. Let $p_1\times m_1,p_2\times m_2,\dots,p_k\times m_k$ be the sizes of the blocks in the KCF of $A-\lambda C$. It follows from the arguments in~\cite[Section XII.3]{gantmacher} that $M_\varepsilon(A,C)$ has a nontrivial kernel if and only if there is a polynomial vector $x(\lambda)$ of degree strictly smaller than $\varepsilon$ such that $(A-\lambda C)x(\lambda)=0$, or, equivalently, if and only if the KCF of $A-\lambda C$ contains a block $L_{\varepsilon_2}$ with $\varepsilon_2<\varepsilon$. Hence, if $M_\varepsilon(A,C)$ is invertible, then the KCF of $A-\lambda C$ contains only Jordan blocks, together with blocks $L_\eta^\top$, and blocks $L_{\varepsilon_2}$ with $\varepsilon_2\geq \varepsilon$. For each of these blocks, one can check that $m_i/p_i \leq (\varepsilon+1)/\varepsilon$, for $i=1,\hdots,k$, and the equality holds only for blocks of type $L_\varepsilon$. In particular, we have
\[
p(\varepsilon+1) = \sum_{i=1}^k p_i(\varepsilon+1) \geq \sum_{i=1}^k m_i\varepsilon = m\varepsilon.
\]
Since $M_\varepsilon(A,C)$ is square, equality must hold for all $i=1,\hdots,k$, which means that all blocks are of type $L_\varepsilon$.

Now let us prove the other direction: suppose that the KCF of $A-\lambda C$ consists only of blocks $L_\varepsilon$. Then, $p=k\varepsilon$ and $m=k(\varepsilon+1)$, so $p(\varepsilon+1)=m\varepsilon$, and $M_\varepsilon(A,C)$ is square. Moreover, again by the same reasonings in~\cite[Section XII.3]{gantmacher} as above, $M_\varepsilon(A,C)$ has trivial kernel, so it is invertible.

The second claim in the statement follows by applying the first statement to $A^\top - \lambda C^\top$.
\end{proof}

Using Lemma \ref{lem_acmatrix} we can give an alternative version of the \textbf{UR} characterization in Theorem~\ref{thm:kosir_characterization} that does not involve the KCF explicitly. 

\begin{theorem}
Let $A,C\in\CC^{p\times m}$ and $B,D\in\CC^{n\times q}$. Equation~\eqref{eq:gensylvr} has \emph{exactly} one solution, for any right-hand side $E$, if and only if one of the following situations hold:
	\begin{itemize}
		\item $p=m,q=n$, the pencils $A-\lambda C$ and $D^\top - \lambda B^\top$ are regular and have no common eigenvalues, \emph{or}
		\item $p<m,n<q$, $s=p/(m-p)=n/(q-n)$ is a positive integer, and the square matrices $M_s(A,C)$ and $M_s(B,D)$ are both invertible, \emph{or}
		\item $p>m,n>q$, $s=m/(p-m)=q/(n-q)$ is a positive integer, and the square matrices $M_s(A^\top,C^\top)$ and $M_s(B^\top,D^\top)$ are both invertible.
	\end{itemize}
\end{theorem}
\begin{proof}
The proof is an immediate consequence of Theorem \ref{thm:kosir_characterization}, part {\bf UR}, and Lemma \ref{lem_acmatrix}, just taking into account that if the KCF of an $m\times n$ pencil (respectively, the KCF of an $n\times m$ pencil) is a direct sum of $k$ blocks $L_s$ (resp., $L_s^\top)$, for some fixed $s$, then, as we have seen in the proof of Lemma \ref{lem_acmatrix}, it must be $p=ks, m=k(s+1)$ (resp., $p=k(s+1),m=ks)$, which implies $m>p$ and $s=p/(m-p)$ (resp., $m<p$ and $s=m/(p-m)$).
\end{proof}

\section{Conclusions and open problems}

We have provided necessary and sufficient conditions for the generalized $\star$-Sylvester equation \eqref{gensylv} to have a unique solution for any right-hand side $E$ ({\bf UR}). In particular, the coefficient matrix of the associated linear system must be square, which is equivalent to the condition $mn=pq$, and the problem becomes equivalent to characterizing the uniqueness of solution of the homogeneous equation \eqref{hom}. The characterization that we have obtained extends the recent one in \cite{di16} for the case of square coefficients. We have also reviewed the solution of problems {\bf SR} (solvability for any right-hand side), {\bf OR} (at most one solution for any right-hand side), and {\bf UR} (unique solvability for any right-hand side) for the generalized Sylvester equation \eqref{eq:gensylvr}. 

It is interesting to compare the conditions for unique solvability (\textbf{UR}) for the two equations~\eqref{gensylv} and~\eqref{eq:gensylvr}, given in Theorems~\ref{thm:main} and~\ref{thm:kosir_characterization} since, in the case of rectangular coefficients, there are more significant differences than those in the case of square coefficients. For the generalized Sylvester equation \eqref{eq:gensylvr}, the only additional case with unique solution is when the KCF of the associated pencils $A-\la C$ and $D^\top-\la B^\top$ contains only certain singular blocks; for the generalized $\star$-Sylvester equation, the spectral properties and Kronecker invariants are not sufficient to determine the answer, and it is necessary to check the invertibility of one of the coefficients.

To our knowledge, small-pencil characterizations for questions~\textbf{US}, \textbf{SR}, and~\textbf{OR} for the generalized $\star$-Sylvester equation are still not present in the literature and arise as a natural open problem to approach in the future.

\bigskip

\noindent{\bf Acknowledgments.} We wish to thank Emre Mengi for pointing out reference \cite{kmnt14} and, as a consequence, references \cite{kosir,kosir-tech} as well. We also thank two anonymous referees for valuable comments that allowed to improve the original version of the paper, as well as one of the referees for pointing out reference \cite{rozsa}.

\bibliographystyle{abbrv}
\bibliography{rectangular}

\end{document}